\documentclass[11pt]{amsart}
\usepackage[T1]{fontenc}
\usepackage{lmodern}
\usepackage[a4paper,margin=27mm]{geometry}
\usepackage{mathtools,amssymb}
\usepackage{microtype}
\usepackage{booktabs}
\usepackage{needspace}
\usepackage[hidelinks]{hyperref}
\hypersetup{
 pdftitle={Simple critical zeros and distinct zeros of the Riemann zeta function in short intervals},
 pdfsubject={Unconditional short-interval proportions of simple critical zeros and distinct zeros},
 pdfkeywords={Riemann zeta function, simple zeros, distinct zeros, short intervals, pair correlation}
}
\allowdisplaybreaks[1]
\numberwithin{equation}{section}
\newtheorem{theorem}{Theorem}[section]
\newtheorem{proposition}[theorem]{Proposition}
\newtheorem{lemma}[theorem]{Lemma}

\theoremstyle{remark}
\newtheorem{remark}[theorem]{Remark}
\newcommand{\R}{\mathbb R}
\newcommand{\C}{\mathbb C}
\newcommand{\Z}{\mathbb Z}

\newcommand{\eps}{\varepsilon}
\newcommand{\supp}{\operatorname{supp}}
\newcommand{\dist}{\operatorname{dist}}
\newcommand{\wh}{\widehat}

\newcommand{\Ns}{N_0^s}
\newcommand{\Nd}{N^d}
\newcommand{\cMT}{C_{\mathrm{MT}}}
\newcommand{\cF}{\mathcal F}
\newcommand{\cZ}{\mathcal Z}
\newcommand{\ind}{\mathbf 1}
\newcommand{\mc}{\mathcal}
\newcommand{\tr}{\operatorname{tr}}

\DeclareMathOperator{\Impart}{Im}

\author[B. Wang]{Biao Wang}
\address{School of Mathematics and Statistics, Yunnan University, Kunming, Yunnan 650500, China}
\email{bwang@ynu.edu.cn}
\date{\today}

\title[Simple and distinct zeros in short intervals]{Simple critical zeros and distinct zeros of the Riemann zeta function in short intervals}

\begin{document}
\begin{abstract}
Recently,  on the non-trivial zeros of the Riemann zeta function, it is obtained by Alp\"oge and Furman that more than 67.25\% of the zeros are simple and on the critical line, and more than 83.62\% are distinct. Later, Lamzouri gave a different and more direct proof. In this article, we will use the method of Lamzouri to give lower bounds on the number of the non-trivial zeros of the Riemann zeta function  in short intervals. To prove the main result, we establish Montgomery's theorem on the pair correlation of zeros of the zeta function in short intervals by following the approach of Baluyot, Goldston, Suriajaya and Turnage-Butterbaugh, and then use Lamzouri's inequality on any finite multiset of complex numbers which is invariant under complex conjugation.  
\end{abstract}
\maketitle

\section{Introduction}

Let $\zeta(s)$ be the Riemann zeta function. In 1859, Riemann conjectured that all the non-trivial zeros of $\zeta(s)$ lie on the critical line $\Re(s)=1/2$. 
In 1914, Hardy \cite{Hardy} proved that infinitely many zeros lie on the critical line, and in 1942 Selberg \cite{Selberg} proved that a positive proportion of the zeros lie on the critical line. In 1947, Min \cite{Min1947} obtained an explicit lower bound for this proportion. In 1974, Levinson \cite{Levinson} introduced a new mollifier method and proved that more than one third lie on the line. In 1989,  Conrey \cite{Conrey} raised this proportion to more than two fifths. Subsequent developments of the mollifier method include the work of  \cite{Feng, Wu2019, PRZZ, CFKLT2025}. 

Besides the mollifier method, a complementary approach uses pair correlation and inequalities for finite collections of zeros. In 1973, Montgomery \cite{Montgomery} initiated this approach and proved, under the Riemann hypothesis, that at least two thirds of the zeros are simple. Montgomery and Taylor \cite{MontgomeryTaylor} subsequently obtained the following larger proportion
\begin{equation}\label{eq:cmt}
2-\cMT
=\frac32-\frac1{\sqrt2}\cot\!\left(\frac1{\sqrt2}\right)
=0.6725\ldots .
\end{equation}
Baluyot, Goldston,
Suriajaya and Turnage-Butterbaugh \cite{BGSTB, BGSTB2025} established an unconditional pair-correlation formula.  Recently, this formula was used to obtain an unconditional proof of  Montgomery's result by Alp\"oge and Furman \cite{AF}. Moreover, they showed that  more than 83.62\% of the nontrivial zeros of  $\zeta(s)$ are distinct. Later, Lamzouri \cite{Lamzouri} gave a different and more direct proof. In the proof, he established an inequality on the number of real simple elements in any finite multiset invariant under complex conjugation. In this article, we consider the simple critical zeros and distinct zeros of the Riemann zeta function in short intervals.

Let $N(T)$ denote the number of nontrivial zeros $\rho=\beta+i\gamma$ of $\zeta(s)$ with $0<\gamma\le T$, counted with multiplicity, and let $N_0(T)$ count those on the critical line $\beta=1/2$, also with multiplicity. Write $\Ns(T)$ for the number of simple zeros on that line up to height $T$. Let $\Nd(T)$ denote the number of distinct nontrivial zeros with $0<\gamma\le T$, each counted once regardless of its multiplicity or real part. For $H>0$ put
\[
 \begin{aligned}
 N(T,H)&=N(T+H)-N(T),\\
 N_0(T,H)&=N_0(T+H)-N_0(T),\\
 \Ns(T,H)&=\Ns(T+H)-\Ns(T),\\
 \Nd(T,H)&=\Nd(T+H)-\Nd(T).
 \end{aligned}
\]

By the Riemann--von Mangoldt formula
\begin{equation}\label{Riemann_von_Mangoldt}
	N(T)=\frac{T}{2\pi}\log\frac{T}{2\pi}-\frac{T}{2\pi}+O(\log T),
\end{equation}
see~\cite[Chapter~8]{Davenport}, we have
\begin{equation}\label{eq:intro-rvm}
 N(T,H)=\frac{H\log T}{2\pi}+O(H+\log T)
 \qquad(1\le H\le T).
\end{equation}

A natural problem is to estimate the proportions $\Ns(T,H)/N(T,H)$ and $\Nd(T,H)/N(T,H)$ for $H=T^\theta$, with $0<\theta<1$ fixed. If all non-trivial zeros are simple and lie on the critical line, then both ratios are equal to \(1\). The history of unconditional results on this problem involves both the length of the interval and the type of zeros detected. Let $N_{\mathrm{odd}}(T)$ count distinct critical-line zeros of odd multiplicity. In the formulations recalled by Karatsuba~\cite[pp.~523--524]{Karatsuba}, Hardy and Littlewood~\cite{HL} proved
\[
N_{\mathrm{odd}}(T+H)-N_{\mathrm{odd}}(T)\gg_\eps H,
\qquad H=T^{1/2+\eps},
\]
for fixed sufficiently small $\eps>0$. Selberg~\cite{Selberg} obtained the additional logarithmic factor,
\begin{equation}\label{eq:selberg-local}
N_{\mathrm{odd}}(T+H)-N_{\mathrm{odd}}(T)
\gg_\eps H\log T,
\qquad H=T^{1/2+\eps}.
\end{equation}
By \eqref{eq:intro-rvm}, this gives a positive proportion of all zeros in each such interval. Karatsuba~\cite{Karatsuba}, in a paper published in Russian in 1984 and in English in 1985, shortened the interval to
\begin{equation}\label{eq:karatsuba-local}
N_{\mathrm{odd}}(T+H)-N_{\mathrm{odd}}(T)
\gg_\eps H\log T,
\qquad H=T^{27/82+\eps}.
\end{equation}

These estimates are stated with an unspecified positive constant. As for explicit proportions, Levinson \cite{Levinson} already proved in 1974 that more than one third of the zeros are on the critical line in intervals of length $T/(\log(T/(2\pi)))^{10}$. This length is larger than $T^\theta$ for every fixed $\theta<1$ as $T\to\infty$. For smaller $\theta$, Steuding~\cite{Steuding} proved in 2002 that
\begin{equation}\label{eq:steuding-local}
\Ns(T,H)\gg H\log T,
\qquad T^{0.552}\le H\le T.
\end{equation}
His argument combines Levinson's method with a mollified mean-square estimate. In our first main result, we will apply Lamzouri's method in \cite{Lamzouri} to improve the range of $H$ in \eqref{eq:steuding-local}. Let 
\[
 c(\theta)=2-\frac{\theta}{2}
       -\frac1{\sqrt2}\cot\!\left(\frac{\theta}{\sqrt2}\right),\quad 0<\theta<1.
\]
Since
\(
 c'(\theta)=\frac12\cot^2(\theta/\sqrt2)>0,
\)
the function $c$ is strictly increasing on $(0,1)$.  It has a unique zero in this interval given by
\[
 \theta_0=0.5501\ldots .
\]

\begin{theorem}\label{thm:main}
For every fixed $\theta_0<\theta<1$, we have
\begin{equation}\label{eq:main}
 \liminf_{T\to\infty}
 \frac{\Ns(T,T^\theta)}{N(T,T^\theta)}
 \ge c(\theta)>0.
\end{equation}
\end{theorem}

For example, taking \(\theta=3/4\) in Theorem~\ref{thm:main}, we find that, for all sufficiently large \(T\), more than \(41.9\%\) of the zeros in \((T,T+T^{3/4}]\), counted with multiplicity, are simple and lie on the critical line. For $\theta\le\theta_0$, that
lower bound is nonpositive and gives no improvement on the trivial bound. Regarding the proportion of distinct zeros of $\zeta(s)$ in \((T,T+T^\theta]\), by extending Lamzouri's inequality, we can get  an explicit positive lower bound for every fixed \(0<\theta<1\).

\begin{theorem}\label{thm:distinct-all}
For any $0<\theta<1$, set
\begin{equation*}
 C_\theta=\frac{\theta}{2}
       +\frac1{\sqrt2}\cot\!\left(\frac{\theta}{\sqrt2}\right).
\end{equation*}
For $C\ge1$, define
\begin{equation}\label{eq:distinct-envelope}
 \mathcal D(C)=\frac{2q+1-C}{q(q+1)},
 \qquad q=\lfloor C\rfloor,
\end{equation}
where $\lfloor C\rfloor := \max\{n\in \Z: n\leq C\}$.  Then, unconditionally, we have
\begin{equation}\label{eq:distinct-all}
 \liminf_{T\to\infty}
 \frac{\Nd(T,T^\theta)}{N(T,T^\theta)}
 \ge\mathcal D(C_\theta)>0
\end{equation}
for  any $0<\theta<1$.
\end{theorem}

For example, taking \(\theta=1/4\) in Theorem~\ref{thm:distinct-all}, we obtain
\(
N^d(T,T^{1/4})>0.2458\,N(T,T^{1/4})
\)
for all sufficiently large \(T\). Moreover, by 
\[
\frac{2q+1-C}{q(q+1)}-\frac1C
=
\frac{(C-q)(q+1-C)}{Cq(q+1)}\ge 0
\]
for $C-1<q=\lfloor C\rfloor \leq C$, we get a simpler, though generally weaker, lower bound
\begin{equation}\label{eq_distinct_all_simple}
 \liminf_{T\to\infty}
 \frac{\Nd(T,T^\theta)}{N(T,T^\theta)}
 \ge \frac1{C_\theta}>0.
\end{equation}

To prove Theorem~\ref{thm:main}, in Section~\ref{sec:local} we will cite Lamzouri's inequality on finite multisets invariant under complex conjugation, and prove Montgomery's pair correlation  formula in short intervals. Then to apply  Lamzouri's inequality, in Section~\ref{sec:unweighted} we will show how to remove the rational weight in the pair correlation  formula. Finally, in Section~\ref{sec:optimization} we will complete the proof of Theorem~\ref{thm:main} by minimizing a functional coming from the removing process. To prove Theorem~\ref{thm:distinct-all}, in Section~\ref{sec:distinct-all} we extend Lamzouri's inequality by introducing an integer parameter $q\ge1$ and then use the argument in the proof of Theorem~\ref{thm:main}.

\subsection*{Notation}

Throughout the proof, we fix
\[
 0<\lambda<\theta<1,\qquad H=T^\theta,\qquad
 L=\log T,\qquad I=(T,T+H].
\]
Unless explicitly stated otherwise, zero sums count multiplicity. Write $\rho=1/2+\delta+i\gamma$, so $|\delta|<1/2$ by the prime number theorem. For a compactly supported integrable function $f$, we define the  Fourier transform by
\[
 \wh f(z)=\int_{\R}f(u)e^{-2\pi izu}\,du,\qquad z\in\C.
\]
The constants in asymptotic estimates may depend on the fixed parameters $\lambda,\theta$ and on the indicated test functions. We use $N(I)=N(T,H)$, $\Ns(I)=\Ns(T,H)$ and $\Nd(I)=\Nd(T,H)$ when convenient.

\section{Pair correlation formula in short intervals}
\label{sec:local}

Let  $\mc{Z}$ be a finite multiset. For an element $z$ of $\mc{Z}$, we denote by $m_z$ its multiplicity.  We say that $\mc{Z}$ is \emph{invariant under complex conjugation} if for all $z\in \mc{Z}$ we have $\overline{z}\in \mc{Z}$ and $m_{\overline{z}}=m_z.$  We cite an inequality of Lamzouri on such invariant multisets.

\begin{proposition}[{\cite[Proposition~2.1]{Lamzouri}}]
\label{prop:finite}
Let $\lambda>0$ and let $\eta\in L^2(\R)$ be real and even, with $\supp(\eta)\subset(-\lambda,\lambda)$ and $\wh{\eta^2}(0)=1$. Let $\cZ$ be a non-empty finite multiset of complex numbers invariant under complex conjugation. For the kernel $K(\xi):=\wh{\eta^2}(\xi)$, we have
\begin{equation}\label{eq:finite}
 \sum_{\substack{z\in\cZ\cap\R\\m_z=1}}1
 \ge 2\sum_{z\in\cZ}1-\sum_{z,s\in\cZ}K(z-s)^2.
\end{equation}
Moreover, the number of distinct elements of $\cZ$ is at least
\begin{equation}\label{eq:finite-distinct}
 \frac32\sum_{z\in\cZ}1
       -\frac12\sum_{z,s\in\cZ}K(z-s)^2.
\end{equation}
\end{proposition}

Put
\begin{equation}\label{eq:weighted-sum}
 w(z)=\frac4{4-z^2},\qquad
 \cF_I(x)=\sum_{\gamma,\gamma'\in I}
              x^{\rho-\rho'}w(\rho-\rho'),\qquad x>0
\end{equation}
for $I=(T,T+H]$ and $H=T^\theta$, $0<\theta<1$. Then
\begin{equation}\label{eq:even-F}
 \cF_I(x^{-1})=\cF_I(x).
\end{equation}
To apply  Proposition~\ref{prop:finite} in the proof of Theorem~\ref{thm:main}, in this section we mainly prove the following short-interval version of Montgomery's theorem on the pair correlation of zeros of the zeta function.

\begin{theorem}\label{thm:local-pair}
Let $0<\lambda<\theta<1$ be fixed, and let
$g\in C_c^\infty(\R)$ be a fixed real even function with
$\supp g\subset[-\lambda,\lambda]$. Then
\begin{equation}\label{eq:local-pair}
 \sum_{\gamma,\gamma'\in I}
   \wh g\left(\frac{i(\rho-\rho')L}{2\pi}\right)
                      w(\rho-\rho')=\frac{HL}{2\pi}
       \left(g(0)+\int_\R|\alpha|g(\alpha)\,d\alpha\right)
       +O_g(H+T^\lambda L^2).
\end{equation}
\end{theorem}

Theorem~\ref{thm:local-pair} is a short-interval analogue of Lemma 5 of Baluyot, Goldston, Suriajaya and Turnage-Butterbaugh's work \cite{BGSTB}.  Its proof uses corresponding short-interval versions of Lemmas~3 and~4 of that paper.

Write
$\rho=1/2+\delta+i\gamma$ and define
\begin{equation}\label{eq:A}
 A(x,t)=\sum_\rho
  \frac{2x^{\delta+i(\gamma-t)}}{1+((t-\gamma)+i\delta)^2}.
\end{equation}
It is absolutely convergent for each fixed $x\ge1$ and real $t$. Indeed, by \cite[(2.14)]{BGSTB}, we have $|A(x,t)|\ll x^{1/2}\log(|t|+2)$. Let $A_I(x,t)$ be the sum restricted to $\gamma\in I$.  The unconditional explicit formula in~\cite[Lemma~1]{BGSTB}, originating in~\cite{Montgomery}, implies that for $x\ge1$ and $t\in I$,
\begin{equation}\label{eq:explicit}
 A(x,t)=-D_x(t)+B_x(t)+E_x(t),
\end{equation}
where
\begin{equation}\label{eq:coefficients}
 D_x(t)=\sum_{n\ge2}a_n n^{-it},\quad
 a_n=\frac{\Lambda(n)}{\sqrt n}
      \min\left(\frac nx,\frac xn\right),\quad
 B_x(t)=\frac{\log(t+2)}x,
\end{equation}
and
\begin{equation}\label{eq:E}
 E_x(t)\ll \frac1x+\frac{x^{1/2}}{T^2}+\frac{x^{-5/2}}T
 \ll\frac1x\qquad(1\le x\le T^\lambda).
\end{equation}
The last inequality uses $x<T$.

\begin{lemma}\label{lem_norm}
	We have
\begin{equation}\label{eq:norm}
 2\pi\cF_I(x)=\int_\R |A_I(x,t)|^2\,dt.
\end{equation}	
\end{lemma}

\begin{proof} By \cite[(2.7)]{BGSTB}, we have 
\[
 \int_\R\frac{du}{(1+u^2)(1+(u+a)^2)}
 =\frac{2\pi}{4+a^2}
\]
for $a\in\C$ and $|\Impart a|<1$. For $\rho=1/2+\delta+i\gamma$ and $\rho'=1/2+\delta'+i\gamma'$, let $d=\rho+\bar\rho'-1$. Set $u=t-\gamma+i\delta$ and
\(
 a=\gamma-\gamma'-i(\delta+\delta')=-id, |\Impart a|=|\delta+\delta'|<1.
\) 
It follows by the above integral that
\[
 \int_\R
 \frac{dt}
 {(1+((t-\gamma)+i\delta)^2)
  (1+((t-\gamma')-i\delta')^2)}
 =\frac{2\pi}{4-d^2}.
\]
Then
\begin{align*}
	\int_\R |A_I(x,t)|^2\,dt &= \sum_{\substack{\rho,\rho' \\ \gamma,\gamma'\in I}} 4x^d \int_\R
 \frac{dt}
 {(1+((t-\gamma)+i\delta)^2)
  (1+((t-\gamma')-i\delta')^2)} \\
  & = 2\pi \sum_{\substack{\rho,\rho' \\ \gamma,\gamma'\in I}} x^d w(d) =  2\pi\cF_I(x).
\end{align*}
The last line follows by replacing  $1-\bar\rho'$ by $\rho'$.
\end{proof}

By Lemma~\ref{lem_norm}, we get that $\cF_I(x)$ is real and nonnegative. Recall that $L=\log T$.

\begin{lemma}\label{lem:localization}
Uniformly for $1\le x\le T^\lambda$,
\begin{equation}\label{eq:localization}
 2\pi\cF_I(x)=\int_I|A(x,t)|^2\,dt+O(xL^3).
\end{equation}
\end{lemma}

\begin{proof}
By \eqref{Riemann_von_Mangoldt}, we have
\begin{equation}\label{number_of_zeros}
	 \sum_{t<\gamma\le t+1}1\ll\log(|t|+3)
 \qquad(t\in\R).
\end{equation}
By $|\delta|<1/2$, we have
\begin{equation}\label{eq:denominator}
 \frac1{|1+((t-\gamma)+i\delta)^2|}
 \le\frac1{(t-\gamma)^2+3/4}
 \ll\frac1{1+(t-\gamma)^2}.
\end{equation}
It follows by \cite[Lemma, Chapter 15]{Davenport} that for $t\in I$, we have
\begin{equation}\label{eq:A-size}
 |A(x,t)|+|A_I(x,t)|\ll\sqrt x\,L.
\end{equation}

Moreover, for $t\in I=(T,T+H]$ we have
\begin{equation}\label{eq:A-tail}
 |A(x,t)-A_I(x,t)|
 \ll\sqrt x\,L\left(
       \frac1{1+t-T}+\frac1{1+T+H-t}\right).
\end{equation}

Indeed, put
\(
a=t-T,\qquad b=T+H-t,
\)
then $a,b\ge0$ and $a+b=H$. By \eqref{eq:denominator},
\[
|A(x,t)-A_I(x,t)|
\ll
\sqrt{x}\sum_{\gamma\notin I}
\frac1{1+(t-\gamma)^2}.
\]
First consider zeros with $\gamma\le T$ and $|\gamma|\le3T$. Partition their ordinates into the intervals
\[
(T-j-1,T-j],\qquad j=0,1,2,\ldots.
\]
By \eqref{number_of_zeros}, each interval meeting $[-3T,3T]$ contains $O(L)$ zeros. For an ordinate in the $j$th interval,
\[
t-\gamma\ge t-(T-j)=a+j.
\]
It follows that
\[
\begin{aligned}
\sum_{\substack{\gamma\le T\\|\gamma|\le3T}}
\frac1{1+(t-\gamma)^2}
& = \sum_{j=0}^\infty \sum_{\substack{\gamma\le T, |\gamma|\le3T \\ \gamma \in (T-j-1,T-j]}} 
\frac1{1+(t-\gamma)^2} \\
& \ll
L\sum_{j=0}^{\infty}\frac1{1+(a+j)^2}\\
&\ll \frac{L}{1+a}.
\end{aligned}
\]

Similarly, partition the ordinates with $\gamma>T+H$ and $|\gamma|\le3T$ into
\[
(T+H+j,T+H+j+1],\qquad j=0,1,2,\ldots.
\]
For an ordinate in the $j$th interval,
\[
\gamma-t\ge b+j.
\]
The same unit-interval estimate therefore yields
\[
\sum_{\substack{\gamma>T+H\\|\gamma|\le3T}}
\frac1{1+(t-\gamma)^2}
\ll
L\sum_{j=0}^{\infty}\frac1{1+(b+j)^2}
\ll\frac{L}{1+b}.
\]

It remains to control the zeros with $|\gamma|>3T$. Since $T<t\le T+H\le2T$, these ordinates satisfy
\[
|t-\gamma|\ge\frac{|\gamma|}{3}.
\]
Applying the unit-interval estimate \eqref{number_of_zeros} again, we obtain
\[
\sum_{|\gamma|>3T}\frac1{1+(t-\gamma)^2}
\ll \sum_{|\gamma|>3T}\frac1{\gamma^2}\ll
\sum_{n\ge3T}
\frac{\log(n+3)}{n^2}\ll \frac{L}{T}.
\]

Combining these estimates gives
\begin{equation*}
	|A(x,t)-A_I(x,t)|
\ll
\sqrt{x}\,L
\left(
\frac1{1+a}+\frac1{1+b}+\frac1T
\right).
\end{equation*}

Finally, since $a,b\le H\le T$,
\[
\frac1{1+a}+\frac1{1+b}
\ge\frac2{1+H}
\ge\frac2{1+T}
\ge\frac1T.
\]
Thus the last term is absorbed, and we conclude that \eqref{eq:A-tail} holds.

For $t\notin I$, similar to the proof of  \eqref{eq:A-tail}, we have
\begin{equation}\label{eq:A-exterior}
 |A_I(x,t)|\ll
 \frac{\sqrt x\,L}{1+\dist(t,I)}.
\end{equation}

Equations \eqref{eq:A-size} and \eqref{eq:A-tail} imply
\[
 \int_I\bigl||A(x,t)|^2-|A_I(x,t)|^2\bigr|\,dt
 \ll xL^2\log(2+H)\ll xL^3.
\]
Equation \eqref{eq:A-exterior} gives
$$\int_{\R\setminus I}|A_I(x,t)|^2\,dt \ll xL^2
\left(
\int_{-\infty}^{T}\frac{dt}{(1+T-t)^2}
+
\int_{T+H}^{\infty}
\frac{dt}{(1+t-T-H)^2}
\right)\ll xL^2.$$
Then the conclusion follows from \eqref{eq:norm}.
\end{proof}

\begin{lemma}\label{lem:coeff}
For the coefficients in \eqref{eq:coefficients}, uniformly for $x\ge1$,
\begin{equation}\label{eq:coeff-sums}
 \sum_{n\ge2}a_n^2=\log x+O(1),\qquad
 \sum_{n\ge2}n a_n^2\ll x\log^2(2x).
\end{equation}
\end{lemma}

\begin{proof}
By the prime number theorem,
\[
 \sum_{n\le u}\Lambda(n)^2=u\log u+O(u)\qquad(u\ge2).
\]
By the partial summation the two parts of the first sum in \eqref{eq:coeff-sums} are
\[
 \frac1{x^2}\sum_{n\le x}n\Lambda(n)^2
       =\frac12\log x+O(1),\qquad
 x^2\sum_{n>x}\frac{\Lambda(n)^2}{n^3}
       =\frac12\log x+O(1).
\]
For the second sum, the estimate $\Lambda(n)\le\log n$ gives
\[
 \frac1{x^2}\sum_{n\le x}n^2\Lambda(n)^2
   +x^2\sum_{n>x}\frac{\Lambda(n)^2}{n^2}
 \ll x\log^2(2x).
\]
This proves \eqref{eq:coeff-sums}.
\end{proof}

\begin{proposition}\label{prop:pointwise}
Uniformly for $1\le x\le T^\lambda$,
\begin{equation}\label{eq:F-pointwise}
 \cF_I(x)=\frac{H}{2\pi}\left(\frac{L^2}{x^2}+\log x\right)
 +O\left(H+\frac{HL}{x^2}+\frac{H\sqrt L}{x}
             +xL^3+\frac{L}{\sqrt x}\right).
\end{equation}
\end{proposition}

\begin{proof} We use \eqref{eq:explicit} and Lemma~\ref{lem:localization} to prove \eqref{eq:F-pointwise}. By Lemma~\ref{lem:coeff} and Montgomery--Vaughan's mean-value theorem in \cite[Corollary 3]{MV}, we have
\begin{equation}\label{eq:D-mean}
 \int_I|D_x(t)|^2\,dt
 =H\sum_{n\ge2}a_n^2+O\left(\sum_{n\ge2}n a_n^2\right)
 =H\log x+O\bigl(H+x\log^2(2x)\bigr).
\end{equation}

For the other main term,
\begin{equation}\label{eq:B-mean}
 \int_I|B_x(t)|^2\,dt
 =\frac{HL^2}{x^2}+O\left(\frac{HL}{x^2}\right),
\end{equation}
since $\log(t+2)=L+O(1)$ on $I$. For the cross term with $D_x$, integration by parts gives
\[
 \left|\int_I\log(t+2)n^{-it}\,dt\right|\ll\frac{L}{\log n}.
\]
Also, $\Lambda(n)/\log n\le1$ implies
\[
 \sum_{n\ge2}\frac{a_n}{\log n}
 \le\frac1x\sum_{n\le x}\sqrt n
       +x\sum_{n>x}n^{-3/2}
 \ll\sqrt x.
\]
Consequently,
\begin{align}
	\left|\int_I B_x(t)\overline{D_x(t)}\,dt\right| & =\frac1{x} \left|\sum_{n\ge2} a_n \int_I \log(t+2)n^{-it} \,dt\right|  \nonumber\\
	& \ll \frac{L}{x}\sum_{n\ge2}\frac{a_n}{\log n} \ll\frac{L}{\sqrt x}. \label{eq:BD}
\end{align}

Equation \eqref{eq:E} gives $\|E_x\|_{L^2(I)}\ll\sqrt H/x$, and \eqref{eq:D-mean} implies $\|D_x\|_{L^2(I)}\ll\sqrt{HL}$ for all sufficiently large $T$, uniformly over $1\le x\le T^\lambda$. Cauchy--Schwarz therefore bounds the $D_x,E_x$ cross term by $O(H\sqrt L/x)$, the $B_x,E_x$ cross term by $O(HL/x^2)$. Combining these estimates yields
\[
 \int_I|A(x,t)|^2\,dt
 =\frac{HL^2}{x^2}+H\log x+O\left(H+\frac{HL}{x^2}
       +\frac{H\sqrt L}{x}+xL^2+\frac{L}{\sqrt x}\right).
\]
Then \eqref{eq:F-pointwise} follows by Lemma~\ref{lem:localization}.
\end{proof}

\begin{proof}[Proof of Theorem~\ref{thm:local-pair}]
Recall that
\[
0<\lambda<\theta<1,\qquad
H=T^\theta,\qquad L=\log T,\qquad I=(T,T+H],
\]
where $\lambda$ and $\theta$ are fixed. By the definition of the Fourier transform,
\[
\widehat g(z)
=\int_{\mathbb R}g(\alpha)e^{-2\pi i\alpha z}\,d\alpha.
\]
Then
\[
\widehat g\left(\frac{i(\rho-\rho')L}{2\pi}\right)
=\int_{-\lambda}^{\lambda}
g(\alpha)e^{\alpha(\rho-\rho')L}\,d\alpha=\int_{-\lambda}^{\lambda}
g(\alpha)T^{\alpha(\rho-\rho')}\,d\alpha
\]
for each pair of zeros $\rho,\rho'$ with ordinates in $I$. Interchanging the finite sum and the integral, by \eqref{eq:even-F}, we obtain an exact identity
\begin{equation}
	 \sum_{\gamma,\gamma'\in I}
   \wh g\left(\frac{i(\rho-\rho')L}{2\pi}\right)
       w(\rho-\rho')
 =\int_{-\lambda}^{\lambda}g(\alpha)\cF_I(T^\alpha)\,d\alpha = 2\int_0^\lambda
g(\alpha)\mathcal F_I(T^\alpha)\,d\alpha.
\end{equation}

We now apply Proposition~\ref{prop:pointwise}. Substituting $x=T^\alpha=e^{L\alpha}$ into \eqref{eq:F-pointwise}, we obtain, uniformly for $0\le\alpha\le\lambda$,
\[
\mathcal F_I(T^\alpha)
=\frac{H}{2\pi}
\left(L^2e^{-2L\alpha}+L\alpha\right)+O\left(
H+HL e^{-2L\alpha}
+H\sqrt L\,e^{-L\alpha}
+L^3e^{L\alpha}
+Le^{-L\alpha/2}
\right).
\]
We evaluate the two main terms and then estimate the integrated error.

The first main term contributes
\[
M_0=\frac{H}{2\pi}
\left(
2L^2\int_0^\lambda
g(\alpha)e^{-2L\alpha}\,d\alpha
\right).
\]
Since $g(\alpha)=0$ for $\alpha>\lambda$,
\[
\begin{aligned}
2L^2\int_0^\lambda
g(\alpha)e^{-2L\alpha}\,d\alpha
&=2L^2g(0)\int_0^\infty e^{-2L\alpha}\,d\alpha+2L^2\int_0^\infty
\bigl(g(\alpha)-g(0)\bigr)e^{-2L\alpha}\,d\alpha\\
&= Lg(0) +2L^2\int_0^\infty
\bigl(g(\alpha)-g(0)\bigr)e^{-2L\alpha}\,d\alpha.
\end{aligned}
\]
By the mean value theorem,
\[
|g(\alpha)-g(0)|
\le \|g'\|_\infty\alpha,
\qquad \alpha\ge0.
\]
Hence the absolute value of the second term is at most
\[
2L^2\|g'\|_\infty
\int_0^\infty\alpha e^{-2L\alpha}\,d\alpha
=\frac12\|g'\|_\infty.
\]
It follows that
\[
2L^2\int_0^\lambda
g(\alpha)e^{-2L\alpha}\,d\alpha
=Lg(0)+O_g(1),
\]
and therefore
\[
M_0=\frac{HL}{2\pi}g(0)+O_g(H).
\]

The second main term contributes
\[
M_1=\frac{HL}{\pi}
\int_0^\lambda\alpha g(\alpha)\,d\alpha.
\]
Using the evenness and support of $g$, we have
\[
2\int_0^\lambda\alpha g(\alpha)\,d\alpha
=
\int_{\mathbb R}|\alpha|g(\alpha)\,d\alpha.
\]
Thus
\[
M_1=\frac{HL}{2\pi}
\int_{\mathbb R}|\alpha|g(\alpha)\,d\alpha.
\]

To estimate the error, we take absolute values and use $|g(\alpha)|\le\|g\|_\infty$. The five terms above satisfy
\[
\begin{aligned}
\int_0^\lambda H\,d\alpha
&\ll H,\\
\int_0^\lambda HL e^{-2L\alpha}\,d\alpha
&\le \frac H2,\\
\int_0^\lambda H\sqrt L\,e^{-L\alpha}\,d\alpha
&\le \frac H{\sqrt L},\\
\int_0^\lambda L^3e^{L\alpha}\,d\alpha
&=L^2(T^\lambda-1)
\ll T^\lambda L^2,\\
\int_0^\lambda Le^{-L\alpha/2}\,d\alpha
&\le2.
\end{aligned}
\]
Consequently, their total contribution is
\[
O_g\left(
H+\frac H{\sqrt L}+T^\lambda L^2+1
\right)
=O_g(H+T^\lambda L^2).
\]

Combining the evaluations of $M_0$ and $M_1$ with the error estimate gives  \eqref{eq:local-pair}, as desired. 
\end{proof}

\section{Removing the rational weight}\label{sec:unweighted}

Choose a fixed real even function
\begin{equation}\label{eq:eta}
 \eta\in C_c^\infty((-\lambda/2,\lambda/2)),\qquad
 \int_\R\eta(u)^2\,du=1,
\end{equation}
and put
\[
 f=\eta^2,\qquad K=\wh f,\qquad Q=f*f.
\]
Both $Q$ and $Q''$ are real even smooth functions supported in $[-\lambda,\lambda]$, and $\wh Q=K^2$. Integration by parts gives $\wh{Q''}(z)=-4\pi^2z^2K(z)^2$. For $z=i(\rho-\rho')L/(2\pi)$, it follows that
\begin{equation}\label{eq:remove-weight}
 K(z)^2=
 \left(\wh Q(z)-\frac{\wh{Q''}(z)}{4L^2}\right)
 w(\rho-\rho').
\end{equation}

Let
$$S_K(I):=\sum_{\gamma,\gamma'\in I}
 K\left(\frac{i(\rho-\rho')L}{2\pi}\right)^2.$$
Following the proof of \cite[Lemma~3.2]{Lamzouri}, we apply Theorem~\ref{thm:local-pair} separately to the fixed test functions \(Q\) and \(Q''\). This yields an asymptotic formula for \(S_K(I)\) without the rational weight $w(\rho-\rho')$.

\begin{lemma}
We have
\begin{equation}\label{eq:unweighted}
 S_K(I)
 =\bigl(C(f)+o_f(1)\bigr)\frac{HL}{2\pi},
\end{equation}
where
\begin{equation}\label{eq:functional}
 C(f)=\int_\R f(u)^2\,du
       +\iint_{\R^2}|u-v|f(u)f(v)\,du\,dv.
\end{equation}
\end{lemma}

\begin{proof}
For a real even test function $g$, write
\[
\mathcal W_I(g)
=
\sum_{\gamma,\gamma'\in I}
\widehat g\left(\frac{i(\rho-\rho')L}{2\pi}\right)
w(\rho-\rho'),
\]
and put
\[
\mathcal M(g)
=
g(0)+\int_{\mathbb R}|\alpha|g(\alpha)\,d\alpha.
\]
Theorem~\ref{thm:local-pair} states that, for each fixed
$g\in C_c^\infty(\mathbb R)$ supported in
$[-\lambda,\lambda]$,
\[
\mathcal W_I(g)
=
\frac{HL}{2\pi}\mathcal M(g)
+O_g(H+T^\lambda L^2).
\]

Both $Q$ and $Q''$ are real, even, smooth functions supported in $[-\lambda,\lambda]$. Summing \eqref{eq:remove-weight} over all pairs of zeros with ordinates in $I$ gives 
\[
S_K(I)
=
\mathcal W_I(Q)-\frac1{4L^2}\mathcal W_I(Q'').
\]
We apply Theorem~\ref{thm:local-pair} separately to the two fixed functions $Q$ and $Q''$.  We obtain
\begin{equation}\label{S_K_I}
	S_K(I)
=\frac{HL}{2\pi}\mathcal M(Q)
-\frac{H}{8\pi L}\mathcal M(Q'')+O_f(H+T^\lambda L^2).
\end{equation}

We first calculate $\mathcal M(Q)$. By the evenness of $f$,
\[
Q(0)
=\int_{\mathbb R}f(u)f(-u)\,du
=\int_{\mathbb R}f(u)^2\,du.
\]
Moreover, Fubini's theorem and the substitution
$v=\alpha-u$ yield
\[
\begin{aligned}
\int_{\mathbb R}|\alpha|Q(\alpha)\,d\alpha
&=\int_{\mathbb R}\int_{\mathbb R}
|\alpha|f(u)f(\alpha-u)\,du\,d\alpha\\
&=\int_{\mathbb R}\int_{\mathbb R}
|u+v|f(u)f(v)\,du\,dv\\
&=\int_{\mathbb R}\int_{\mathbb R}
|u-v|f(u)f(v)\,du\,dv.
\end{aligned}
\]
The last equality follows by replacing $v$ with $-v$ and using $f(-v)=f(v)$. Therefore,
\[
\mathcal M(Q)
=
\int_{\mathbb R}f(u)^2\,du
+\iint_{\mathbb R^2}|u-v|f(u)f(v)\,du\,dv
=C(f).
\]

Since $Q''$ is fixed, we have $\mathcal M(Q'')=O_f(1)$. More explicitly, we have
\[
\mathcal M(Q'')
=
-\int_{\mathbb R}f'(u)^2\,du
+2\int_{\mathbb R}f(u)^2\,du.
\]

Combining these estimates, by \eqref{S_K_I} we find
\[
	S_K(I)
=\frac{HL}{2\pi}C(f)+O_f(H+T^\lambda L^2),
\]
which proves \eqref{eq:unweighted}, since $H=T^\theta, 0<\lambda<\theta<1$ and $L=\log T$.
\end{proof}

Apply Proposition~\ref{prop:finite} to the multiset
\begin{equation}\label{eq:scaled-zero-set}
 \cZ=\left\{\frac{i(\rho-1/2)L}{2\pi}:\gamma\in I\right\},
\end{equation}
where $m_z$ is the multiplicity of the corresponding zero $\rho$. Then
\[
 \sum_{z\in\cZ}1=N(I),\qquad
 \sum_{\substack{z\in\cZ\cap\R\\m_z=1}}1=\Ns(I),\qquad
 \sum_{z,s\in\cZ}K(z-s)^2=S_K(I),
\]
and the number of distinct elements of $\cZ$ is exactly $\Nd(I)$. Both parts of Proposition~\ref{prop:finite} therefore give
\begin{equation}\label{eq:both-counts}
 \Ns(I)\ge2N(I)-S_K(I),\qquad
 \Nd(I)\ge\frac32N(I)-\frac12S_K(I).
\end{equation}

Using \eqref{eq:intro-rvm} and \eqref{eq:unweighted}, we conclude that for every fixed $f=\eta^2$ with $\eta$ satisfying \eqref{eq:eta},
\begin{equation}\label{eq:window-bound}
 \liminf_{T\to\infty}\frac{\Ns(T,T^\theta)}{N(T,T^\theta)}
 \ge2-C(f).
\end{equation}

The lower bound is maximized by minimizing the  functional $C(f)$.

\section{Proof of Theorem~\ref{thm:main}}
\label{sec:optimization}

We now minimize the functional $C(f)$ defined by \eqref{eq:functional}.

\begin{proposition}\label{prop:optimal}
Let $0<\lambda\le1$ and $J=[-\lambda/2,\lambda/2]$. Among real functions $f\in L^2(J)$ with $\int_J f=1$, the functional
\begin{equation*}
 C(f)=\int_\R f(u)^2\,du
       +\iint_{\R^2}|u-v|f(u)f(v)\,du\,dv.
\end{equation*}
defined by \eqref{eq:functional}, with $f$ extended by zero, has the unique minimizer
\begin{equation}\label{eq:cosine}
 f_\lambda(u)=
 \frac{\cos(\sqrt2u)}
      {\sqrt2\sin(\lambda/\sqrt2)}\ind_J(u).
\end{equation}
Its value is
\begin{equation}\label{eq:C-lambda}
 C_\lambda=C(f_\lambda)
 :=\frac{\lambda}{2}
       +\frac1{\sqrt2}\cot\!\left(\frac{\lambda}{\sqrt2}\right).
\end{equation}
The same infimum is obtained in the smaller class $f=\eta^2$ from \eqref{eq:eta}.
\end{proposition}

\begin{proof}
Put $a=\lambda/2$. The function $f_\lambda$ is even and positive on $J$, and direct integration gives $\int_J f_\lambda=1$. For $|u|<a$, set
\[
 G(u)=f_\lambda(u)+\int_{-a}^a|u-v|f_\lambda(v)\,dv.
\]
Since $G''=f_\lambda''+2f_\lambda=0$ and $G$ is even, $G$ is constant. At the right endpoint, taking the interior limit and using the zero first moment of $f_\lambda$, we obtain
\[
 G=f_\lambda(a^-)+a
   =\frac1{\sqrt2}\cot\!\left(\frac{\lambda}{\sqrt2}\right)+\frac{\lambda}{2}.
\]
Integration against $f_\lambda(u)\,du$ proves \eqref{eq:C-lambda}.

We next prove global minimality and uniqueness. Let $f\in L^2(J)$ be any real function satisfying $\int_J f=1$, and put
\[
h=f-f_\lambda.
\]
Then $\int_Jh=0$. Expanding the quadratic functional and using the symmetry of the kernel gives
\[
\begin{aligned}
C(f)-C(f_\lambda)
={}&2\int_J h(u)
\left(
f_\lambda(u)+\int_J|u-v|f_\lambda(v)\,dv
\right)\,du\\
&+\int_Jh(u)^2\,du
+\iint_{J^2}|u-v|h(u)h(v)\,du\,dv.
\end{aligned}
\]
The expression in parentheses is the constant $C_\lambda$. Therefore the linear term vanishes, and
\[
C(f)-C_\lambda
=
\int_Jh(u)^2\,du
+\iint_{J^2}|u-v|h(u)h(v)\,du\,dv.
\]

To evaluate the second term, define
\[
P(u)=\int_{-a}^{u}h(v)\,dv.
\]
Since $h\in L^2(J)$, the function $P$ is absolutely continuous, with $P'=h$ almost everywhere. Moreover,
\[
P(-a)=0,\qquad
P(a)=\int_Jh(v)\,dv=0.
\]
Thus $P\in H_0^1((-a,a))$, the Sobolev space with zero boundary values.

Set
\[
V_h(u)=\int_J|u-v|h(v)\,dv.
\]
Then
\[
V_h'(u)
=
\int_{-a}^{u}h(v)\,dv
-\int_u^{a}h(v)\,dv.
\]
Since $\int_Jh=0$, this simplifies to
\[
V_h'(u)=2P(u).
\]
Integration by parts, using the vanishing boundary values of $P$,  yields
\[
\begin{aligned}
\iint_{J^2}|u-v|h(u)h(v)\,du\,dv
&=\int_Jh(u)V_h(u)\,du\\
&=\int_JP'(u)V_h(u)\,du\\
&=[P(u)V_h(u)]_{-a}^{a}
-\int_JP(u)V_h'(u)\,du\\
&=-2\int_JP(u)^2\,du.
\end{aligned}
\]
 We have therefore proved the following identity
\[
C(f)-C_\lambda
=\int_Jh(u)^2\,du-2\int_JP(u)^2\,du.
\]

The Dirichlet Poincar\'e inequality on an interval of length $\lambda$ (see \cite[Theorem~257, p.~185]{HLP}) gives
\[
\int_JP(u)^2\,du
\le\frac{\lambda^2}{\pi^2}
\int_JP'(u)^2\,du
=\frac{\lambda^2}{\pi^2}\int_Jh(u)^2\,du.
\]
Consequently,
\[
C(f)-C_\lambda
\ge
\left(1-\frac{2\lambda^2}{\pi^2}\right)
\|f-f_\lambda\|_{L^2(J)}^2.
\]
Since $0<\lambda\le1$, the coefficient is strictly positive. Thus $C(f)\ge C_\lambda$, with equality only when $f=f_\lambda$ almost everywhere. 

Finally, we show that the same infimum is obtained in the smaller class specified in \eqref{eq:eta}. Choose real even functions
\[
\chi_\varepsilon\in C_c^\infty((-a,a)),
\qquad 0\le\chi_\varepsilon\le1,
\]
such that $\chi_\varepsilon=1$ on $[-a+\varepsilon,a-\varepsilon]$, where $0<\varepsilon<a$. Put
\[
b_\varepsilon
=\int_J\chi_\varepsilon(u)^2f_\lambda(u)\,du,
\]
and define
\[
\eta_\varepsilon(u)
=
\frac{\chi_\varepsilon(u)\sqrt{f_\lambda(u)}}
{\sqrt{b_\varepsilon}},
\qquad
f_\varepsilon=\eta_\varepsilon^2.
\]
The square root is smooth on $(-a,a)$ because $f_\lambda$ is strictly positive there. Hence $\eta_\varepsilon$, extended by zero, is real, even, and belongs to $C_c^\infty((-a,a))$. Its normalization gives
\[
\int_J\eta_\varepsilon(u)^2\,du=1.
\]

By dominated convergence,
\(
b_\varepsilon\rightarrow1,
f_\varepsilon\rightarrow f_\lambda
\text{ in }L^1(J)\text{ and }L^2(J).
\)
In particular,
\[
\int_Jf_\varepsilon(u)^2\,du
\longrightarrow\int_Jf_\lambda(u)^2\,du.
\]
For the double integral, the bound $|u-v|\le\lambda$ on $J^2$ gives
\[
\begin{aligned}
&\left|
\iint_{J^2}|u-v|
\bigl(f_\varepsilon(u)f_\varepsilon(v)
-f_\lambda(u)f_\lambda(v)\bigr)\,du\,dv
\right|\\
&\qquad\le
\lambda\|f_\varepsilon-f_\lambda\|_{L^1(J)}
\left(
\|f_\varepsilon\|_{L^1(J)}
+\|f_\lambda\|_{L^1(J)}
\right)\\
&\qquad=
2\lambda\|f_\varepsilon-f_\lambda\|_{L^1(J)}
\longrightarrow0.
\end{aligned}
\]

 Therefore,
\(
C(f_\varepsilon)\rightarrow C(f_\lambda)=C_\lambda,
\)
as $\varepsilon\rightarrow0$. This shows that the infimum in the smaller class specified in \eqref{eq:eta} is exactly $C_\lambda$.
\end{proof}

\Needspace{9\baselineskip}
\begin{proof}[Proof of Theorem~\ref{thm:main}]
Fix $\lambda<\theta$ and a sufficiently small positive $\eps>0$ as in the last proof. Apply \eqref{eq:window-bound}  to $f_\eps$ and first let
$T\to\infty$, with both parameters fixed. Next let $\eps\rightarrow0$. We obtain
\[
 \liminf_{T\to\infty}\frac{\Ns(T,T^\theta)}{N(T,T^\theta)} \ge2-C_\lambda.
\]
Finally let $\lambda\rightarrow\theta$ and use \eqref{eq:C-lambda}. This proves \eqref{eq:main}.
\end{proof}

\section{Proof of Theorem~\ref{thm:distinct-all}}
\label{sec:distinct-all}

We extend the distinct-element inequality \eqref{eq:finite-distinct} in Proposition~\ref{prop:finite} to Proposition~\ref{lem:integer-hierarchy} as follows. The $q=1$ case  in Proposition~\ref{lem:integer-hierarchy} recovers \eqref{eq:finite-distinct}.

Let $\eta\in L^2(\R)$ be real, even and compactly supported, with $\int_\R\eta(u)^2\,du=1$, and put $K=\wh{\eta^2}$. Let $\cZ$ be a non-empty finite multiset invariant under complex conjugation. Write $N$ for its total multiplicity, $D$ for its number of distinct elements. Throughout this section, set
\[
 S=\sum_{z,s\in\cZ}K(z-s)^2.
\]

For $z\in\C$, define
\[
 f_z(u):=\eta(u)e^{-2\pi iuz},\qquad
 g_z(u):=\frac{f_z(u)+f_{\bar z}(u)}2,\qquad
 h_z(u):=\frac{f_z(u)-f_{\bar z}(u)}{2i}.
\]
Let $L^2(\R,\C)$ denote the space of complex-valued square-integrable functions on $\R$. Define the inner product on $L^2(\R,\C)$ by $\langle f,g\rangle=\int_{-\lambda}^{\lambda} f(u)\overline{g(u)}\,du$. Then by \cite[(2.9)-(2.11)]{Lamzouri}, we have
\begin{equation}\label{eqn_inner_product}
	\langle f_z,f_s\rangle=K(z-\bar s),\qquad
 \|f_x\|^2=1\ (x\in\R),\qquad
 \|g_z\|^2-\|h_z\|^2=1.
\end{equation}

Let $\mathcal R$ be the set of distinct real elements of $\cZ$, and write its set of distinct non-real elements as
\[
 \mathcal S=\{z_1,\bar z_1,\ldots,z_k,\bar z_k\}.
\]
On the finite-dimensional complex space
\[
 W=\operatorname{Span}_{\C}
 \bigl(\{f_x:x\in\mathcal R\}\cup
       \{g_{z_j},h_{z_j}:1\le j\le k\}\bigr),
\]
define the self-adjoint operator
\begin{equation}\label{operator_A}
 \mathsf A
 =\sum_{x\in\mathcal R}m_xf_x\otimes f_x
  +2\sum_{j=1}^km_{z_j}
        (g_{z_j}\otimes g_{z_j}-h_{z_j}\otimes h_{z_j}),
\end{equation}
where $v\otimes w$ maps $\xi$ to $\langle\xi,w\rangle v$ for $v,w,\xi\in W$. Conjugation invariance of $\cZ$ then gives
\[
 \mathsf A=\sum_{z\in\mathcal R\cup\mathcal S}
                         m_zf_z\otimes f_{\bar z}.
\]
Notice that for rank-one operators $v\otimes w$, we have
\[
 \tr(v\otimes w)=\langle v,w\rangle,\qquad
 \tr\bigl((v\otimes w)(v'\otimes w')\bigr)
 =\langle v',w\rangle\langle v,w'\rangle.
\]
Consequently, by \eqref{eqn_inner_product} each term of $\tr\mathsf A$ equals $m_zK(0)=m_z$, and the $(z,s)$ term of $\tr(\mathsf A^2)$ equals $m_zm_sK(s-z)K(z-s)=m_zm_sK(z-s)^2$. Here evenness of $\eta^2$ implies evenness of $K$. Thus, we obtain
\begin{equation}\label{eq:operator-model}
 \tr\mathsf A=N,\qquad
 \tr(\mathsf A^2)=S.
\end{equation}
In particular, $S$ is real and nonnegative, since $\mathsf A$ is self-adjoint.

By the classical trace--rank inequality
(e.g., see \cite[Lemma~2.2]{Alon2009}), we have
\[
(\tr \mathsf A)^2
\le \operatorname{rank}( \mathsf A)\tr(\mathsf A^2).
\]
Clearly, $ \operatorname{rank}(\mathsf A) \leq D$. Thus, by \eqref{eq:operator-model}, we get that
\begin{equation}\label{trace_rank_inequality}
N^2 \leq DS.
\end{equation}
By the argument in the proof of \cite[Theorem 1.1]{Lamzouri}, we obtain a weaker but positive result than 83.62\%:
\begin{equation}\label{critical_zeros_week_bound}
\liminf_{T\to\infty}\frac{N^d(T)}{N(T)}
\ge\frac{1}{C_1}
=0.7532\ldots.
\end{equation}

Moreover, by \eqref{eq:intro-rvm} and \eqref{eq:unweighted}, it follows from \eqref{trace_rank_inequality}  that
\[
 \liminf_{T\to\infty}\frac{\Nd(T,T^\theta)}{N(T,T^\theta)}
 \ge\frac1{C_\theta}.
\]
This gives a short and direct proof of \eqref{eq_distinct_all_simple}, which suffices to deduce that $\Nd(T,T^\theta) \gg N(T,T^\theta)$ for any $0<\theta <1$. To prove the better estimate \eqref{eq:distinct-all},  we prove a better inequality than \eqref{trace_rank_inequality}.

\begin{proposition}\label{lem:integer-hierarchy}
For every integer $q\ge1$, we have
\begin{equation}\label{eq:integer-hierarchy}
 D\ge\frac{2q+1}{q(q+1)}N-\frac1{q(q+1)}S.
\end{equation}
\end{proposition}

\begin{proof}
Fix an integer $q\ge1$. Split $\mathcal R$ into
\[
 \mathcal R_1:=\{x\in\mathcal R:m_x\le q\},\qquad
 \mathcal R_2:=\{x\in\mathcal R:m_x\ge q+1\},
\]
and write $n=|\mathcal R_1|$, $r=|\mathcal R_2|$. Within this proof, put
\[
 L:=\sum_{x\in\mathcal R_1}m_x,\qquad
 H:=\sum_{x\in\mathcal R_2}m_x,\qquad
 M:=2\sum_{j=1}^km_{z_j}.
\]
Thus $N=L+H+M$ and $D=n+r+2k$.

Consider the nested complex subspaces $U\subseteq V$ of $W$
\[
 \begin{aligned}
 U&:=\operatorname{Span}_{\C}
       \bigl(\{g_{z_j}:1\le j\le k\}
                 \cup\{f_x:x\in\mathcal R_2\}\bigr),\\
 V&:=\operatorname{Span}_{\C}
       \bigl(U\cup\{f_x:x\in\mathcal R_1\}\bigr),
 \end{aligned}
\]
and denote their orthogonal projections by $P_U$ and $P_V$.

First, we estimate  $\tr(P_U\mathsf A)$  for the operator $\mathsf A$ defined by \eqref{operator_A}.  By the definition of the operator $v\otimes w$, $(v\otimes w)\xi=\langle \xi,w\rangle v$, for any $v\in W$ we have
\[
 (P_U(v\otimes v)) \xi
 =P_U ((v\otimes v) \xi ) =P_U( \langle\xi,v\rangle v )= \langle\xi,v\rangle P_Uv = ((P_Uv)\otimes v) \xi,
\]
so $P_U(v\otimes v)=(P_Uv)\otimes v$. Therefore,
\begin{equation}\label{projection_U_norm}
 \tr\bigl(P_U(v\otimes v)\bigr)
 =\langle P_Uv,v\rangle
 =\|P_Uv\|^2.
\end{equation}

By \eqref{projection_U_norm}, we have $\tr (P_U(f_x\otimes f_x)) \ge0$ for $x\in\mathcal R_1$. For $x\in\mathcal R_2$,  we have $P_Uf_x=f_x$  and $\| f_x\|^2=1$ by \eqref{eqn_inner_product},  so  $\tr (P_U(f_x\otimes f_x)) = \|P_U f_x\|^2=  \| f_x\|^2=1$.  For the conjugate pair $z_j,\bar z_j$, let
\[
 \mathsf A_j
 =2m_{z_j}\bigl(g_{z_j}\otimes g_{z_j}
                 -h_{z_j}\otimes h_{z_j}\bigr).
\]
Then
\[
 \begin{aligned}
 \tr(P_U\mathsf A_j)
 &=2m_{z_j}\left(
     \tr\bigl(P_U(g_{z_j}\otimes g_{z_j})\bigr)
     -\tr\bigl(P_U(h_{z_j}\otimes h_{z_j})\bigr)
     \right)\\
 &=2m_{z_j}\bigl(
      \|P_Ug_{z_j}\|^2-\|P_Uh_{z_j}\|^2
      \bigr).
 \end{aligned}
\]
For $g_{z_j}\in U$, we have $P_Ug_{z_j}=g_{z_j}$. Also, $\|P_Uh_{z_j}\|^2 \leq \|h_{z_j}\|^2$. By \eqref{eqn_inner_product}, we have $\|g_{z_j}\|^2-\|h_{z_j}\|^2=1$.  Consequently, 
 \[
 \tr(P_U\mathsf A_j)
 \ge 2m_{z_j}\big(
      \|g_{z_j}\|^2-\|h_{z_j}\|^2
      \big)
 = 2m_{z_j}.
\]
Thus, the above estimates yield
\begin{equation}\label{lower_bound_tr_P_UA}
 \tr(P_U\mathsf A)\ge H+M.
\end{equation}

Second, we estimate  $\tr(P_V\mathsf A)$. Let $V^\perp$ be the orthogonal complement of $V$ in $W$. Let $P_{V^\perp}:=I_W-P_V$ be the orthogonal projection onto $V^\perp$, where $I_W$ is the identity map. Then  $P_{V^\perp}f_x=0, P_{V^\perp}g_{z_j}=0$ for $x\in\mathcal R$ and $1\leq j\leq k$. Notice that for any orthogonal projection $P$ and any vector $v\in W$, self-adjointness of $P$ gives
\[
(P(v\otimes v)P)\,\xi = P((v\otimes v)(P\,\xi )= P(\langle P\xi,v\rangle v ) = \langle P\xi,v\rangle Pv=\langle\xi,Pv\rangle Pv =((Pv)\otimes(Pv)) \xi.
\]
So
\(
 P(v\otimes v)P=(Pv)\otimes(Pv).
\)
Applying this identity to each term of $\mathsf A$, we obtain
\[
 P_{V^\perp}\mathsf A P_{V^\perp}
 =-2\sum_{j=1}^k m_{z_j}
       (P_{V^\perp}h_{z_j})\otimes(P_{V^\perp}h_{z_j}).
\]

Similar to \eqref{projection_U_norm}, we have $ \tr\bigl(P_V(v\otimes v)\bigr) =\|P_V v\|^2$, so
\[
 \tr\bigl(
   P_{V^\perp}\mathsf A P_{V^\perp}
 \bigr)
 =-2\sum_{j=1}^k m_{z_j}\|P_{V^\perp}h_{z_j}\|^2.
\]

Since $P_{V^\perp}^2=P_{V^\perp}$, we have
\(
 \tr(P_{V^\perp}\mathsf A)
 =\tr\bigl(
    P_{V^\perp}\mathsf A P_{V^\perp}
   \bigr) \leq 0.
\)
Using $P_V+P_{V^\perp}=I_W$ and $\tr\mathsf A=N$, we therefore obtain
\begin{equation}\label{lower_bound_tr_P_VA}
 \tr(P_V\mathsf A)=\tr\mathsf A
       -\tr(P_{V^\perp}\mathsf A) \ge N.
\end{equation}

Third, for $q\ge1$, define the self-adjoint operator $\mathsf B:=qP_V+P_U$. By \eqref{lower_bound_tr_P_UA} and \eqref{lower_bound_tr_P_VA}, we have
\begin{equation}\label{lower_bound_tr_AB}
\tr(\mathsf A\mathsf B)\ge qN+H+M.
\end{equation}

Let $\|\mathsf B\|_{\mathrm{HS}}^2$ be the Hilbert--Schmidt norm of $\mathsf B$. Since $U\subset V$, we have the orthogonal decomposition
\[
 W=U\oplus(V\cap U^\perp)\oplus V^\perp.
\]
If $\xi\in U$, then $P_U\xi=P_V\xi=\xi$, and hence
\(
 \mathsf B\xi=(q+1)\xi.
\)
If $\xi\in V\cap U^\perp$, then $P_V\xi=\xi$ and $P_U\xi=0$, so
\(
 \mathsf B\xi=q\xi.
\)
Finally, if $\xi\in V^\perp$, then $\xi\in U^\perp$ as well, and therefore $\mathsf B\xi=0$. Choose an orthonormal basis $\{e_j\}_{j=1}^{\dim W}$ adapted to this orthogonal decomposition, then
\[
 \|\mathsf B\|_{\mathrm{HS}}^2
 =\sum_{j=1}^{\dim W}\|\mathsf B e_j\|^2 = (q+1)^2\dim U+q^2\dim (V\cap U^\perp).
\]
By $\dim U\le r+k$ and $\dim (V\cap U^\perp)= \dim V-\dim U\le n$, it follows that
\begin{equation}\label{upper_bound_HS_norm_B}
 \|\mathsf B\|_{\mathrm{HS}}^2
 \le(q+1)^2(r+k)+q^2n.
\end{equation}

Since $\mathsf A$ and $\mathsf B$ are self-adjoint, so is $ \mathsf A - \mathsf B$. Then by $\tr(\mathsf B\mathsf A)=\tr(\mathsf A\mathsf B)$,
\[
 \|\mathsf A-\mathsf B\|_{\mathrm{HS}}^2
 =\tr\bigl((\mathsf A-\mathsf B)^2\bigr) =\tr(\mathsf A^2)
   -2\tr(\mathsf A\mathsf B)
   +\tr(\mathsf B^2) \ge0.
\]
Using $\tr(\mathsf A^2)=S$ from \eqref{eq:operator-model} and $\tr(\mathsf B^2)=\|\mathsf B\|_{\mathrm{HS}}^2$, we therefore obtain
\[
 S\ge2\tr(\mathsf A\mathsf B)
          -\|\mathsf B\|_{\mathrm{HS}}^2.
\]
By \eqref{lower_bound_tr_AB} and \eqref{upper_bound_HS_norm_B},
\[
 \begin{aligned}
 S &\ge2(qN+H+M)-\|\mathsf B\|_{\mathrm{HS}}^2\\
 &\ge2qN+2(H+M)-(q+1)^2(r+k)-q^2n.
 \end{aligned}
\]

Using $N=L+H+M, D=n+r+2k$ and the choice of $\mathcal R_1$ and $\mathcal R_2$, subtracting $(2q+1)N-q(q+1)D$ from the previous lower bound for $S$, we obtain
\[
 \begin{aligned}
 &\,\quad\bigl[2qN+2(H+M)-(q+1)^2(r+k)-q^2n\bigr] -\bigl[(2q+1)N-q(q+1)D\bigr]\\
 &=M+(q^2-1)k+\bigl(H-(q+1)r\bigr)+(qn-L)\\
 &=M+(q^2-1)k
 +\sum_{x\in\mathcal R_2}(m_x-q-1)
 +\sum_{x\in\mathcal R_1}(q-m_x)\ge0.
 \end{aligned}
\]
Consequently,
\[
 S\ge(2q+1)N-q(q+1)D,
\]
which gives \eqref{eq:integer-hierarchy}, as desired.
\end{proof}

\begin{remark}
In \cite{CFKLT2025}, Conrey et al. proved that a short mollifier suffices to establish a positive proportion of the number of critical zeros of $\zeta(s)$. In this remark, we show an inequality weaker than \eqref{eq:finite} that can be used to obtain a positive proportion. With the notation of Proposition~\ref{lem:integer-hierarchy}, let
\(
 n:=\#\{x\in\mathcal R:m_x=1\}
\)
be the number of simple real elements of $\cZ$. Take $q=1$ in the proof of Proposition~\ref{lem:integer-hierarchy}.  Then  by \eqref{lower_bound_tr_P_VA}, we have $N\leq \tr(P_V\mathsf A)$. Clearly, $\dim V \le n+r+k$. Applying the Cauchy--Schwarz inequality to $P_V$ and $\mathsf A$, we have
\begin{align*}
 N^2 \le \bigl|\tr(P_V\mathsf A)\bigr|^2 \le \tr(P_V^*P_V)\tr(\mathsf A^*\mathsf A) =\tr(P_V)\tr(\mathsf A^2) =(\dim V)S\le(n+r+k)S.
\end{align*}
By \(
 N=n+\sum_{x\in\mathcal R_2}m_x
       +2\sum_{j=1}^k m_{z_j}
 \ge n+2r+2k,
\)
we get that
\(
 n+r+k\le\frac{N+n}{2}.
\)
Theorefore, we get that $N^2\le\frac{N+n}{2}S$, which is an analogue of the inequality \eqref{trace_rank_inequality} for simple real elements of $\cZ$.
That is,
\begin{equation}\label{simple_real_inequality}
 \frac{n}N\ge \frac{2N}{S}-1.
\end{equation}
Using \eqref{simple_real_inequality} and the proof of \cite[Theorem 1.1]{Lamzouri}, we obtain that
\begin{equation}\label{critical_zeros_week_bound}
\liminf_{T\to\infty}\frac{N_0^s(T)}{N(T)}
\ge\frac{2}{C_{\mathrm{MT}}}-1
=0.5065\ldots.
\end{equation}
Since Lamzouri's inequality \eqref{eq:finite} is better than \eqref{simple_real_inequality},  the lower bound of \eqref{critical_zeros_week_bound} is weaker than \eqref{eq:cmt}. 
\end{remark}

\begin{proof}[Proof of Theorem~\ref{thm:distinct-all}]
Fix $0<\theta<1$, an integer $q\ge1$, a parameter $0<\lambda<\theta$, and a smooth function $f=\eta^2$ with $\eta$ satisfying \eqref{eq:eta}. Apply Proposition~\ref{lem:integer-hierarchy} to the zero multiset \eqref{eq:scaled-zero-set}. This multiset is invariant under complex conjugation by the symmetry $\rho\mapsto1-\bar\rho$ of the zeros. Its total multiplicity and distinct count are $N(T,T^\theta)$ and $\Nd(T,T^\theta)$, respectively, while $S=S_K(I)$. By \eqref{eq:intro-rvm} and \eqref{eq:unweighted},
\[
 \liminf_{T\to\infty}\frac{\Nd(T,T^\theta)}{N(T,T^\theta)}
 \ge\frac{2q+1-C(f)}{q(q+1)}.
\]
Use the smooth approximations to $f_\lambda$ constructed in Proposition~\ref{prop:optimal}, first keeping each approximation fixed as $T\to\infty$. Then let $\lambda\to\theta$. This gives
\begin{equation}\label{eq:distinct-each-integer}
 \liminf_{T\to\infty}\frac{\Nd(T,T^\theta)}{N(T,T^\theta)}
 \ge\frac{2q+1-C_\theta}{q(q+1)}\qquad(q\ge1).
\end{equation}
To optimize, put $L_q(C)=(2q+1-C)/(q(q+1))$. A direct calculation gives
\[
 L_q(C)-L_{q+1}(C)
 =\frac{2(q+1-C)}{q(q+1)(q+2)}.
\]
Thus the sequence increases until $q=\lfloor C\rfloor$ and decreases thereafter. Since $C_\theta>1$, its maximum over integers $q\ge1$ is precisely $\mathcal D(C_\theta)$ in \eqref{eq:distinct-envelope}. This proves \eqref{eq:distinct-all}.

Indeed, $C_\theta>C_1>1$ for $0<\theta<1$, since $C_\theta'=-\tfrac12\cot^2(\theta/\sqrt2)<0$ and $C_1=1.3274\ldots$. For $q=\lfloor C_\theta\rfloor$, we have $q\le C_\theta<q+1$, and hence $2q+1-C_\theta>q>0$. Thus $\mathcal D(C_\theta)>0$.
\end{proof}

\begin{remark}
In \cite{Lamzouri}, Lamzouri also proved that  more than 88.76\% of the non-trivial zeros of the zeta function  are either simple or lie on the critical line (or both), and  the average of the proportions of simple zeros and of zeros on the critical line is at least 83.62\%. One may combine his arguments with the results in this article to obtain the short-interval analogues.
\end{remark}

\section*{Acknowledgments}
AI tools were used to explore some of the proof strategies and calculations. The author takes full responsibility for human verifications and the content.

\end{document}